\documentclass{article}
\usepackage{graphics}
\usepackage{graphicx}
\usepackage{latexsym}
\usepackage{amssymb}
\usepackage{amsmath}

\newtheorem{theorem}{Theorem}[section]
\newtheorem{lemma}[theorem]{Lemma}
\newtheorem{proposition}[theorem]{Proposition}
\newtheorem{corollary}[theorem]{Corollary}

\newtheorem{preexample}{Example}[section]

\newtheorem{preremark}{Remark}

\newcommand{\qed }{ \hfill $\Box$ }

\newcommand{\g}{\mathfrak{g}}
\newcommand{\mk}{\mathfrak{k}}
\newcommand{\ma}{\mathfrak{a}}
\newcommand{\mn}{\mathfrak{n}}
\newcommand{\p}{\mathfrak{p}}
\newcommand{\s}{\mathfrak{s}}
\renewcommand{\log}[1]{ \mathcal{L}{(#1)}}

\begin{document}

\title{A factorization of harmonic maps into Semisimple Lie groups}
\author{Sim\~{a}o N. Stelmastchuk \thanks{Partially supported by CNPq/Universal grant n$^{\circ}$ $476024/2012-9$} \\ Universidade Federal do Paran\'a \\ Jandaia do Sul, Brazil, simnaos@gmail.com}

\maketitle

% - - - - - - -%

\maketitle

\begin{abstract}
  We factorize harmonic maps with values in a semisimple Lie groups in a product of harmonic maps with values in the components of the Iwasawa decomposition. In particular, we use this factorization to study the harmonic maps from $\mathbb{R}^n$ into $SL(2,\mathbb{R})$.
\end{abstract}


\vspace{0.3cm} \noindent {\bf MSC2010 subject classification:} 58E20, 22E46, 53C43.

\section{Introduction}

The study of harmonic maps on Lie groups and Homogeneous space trough factorization has been extensively used. Among all works that used this technic it is well known those due to K.K. Uhlenbeck \cite{uhlenbeck} and F.E. Burstall and M.A. Guest \cite{burstall}. Since for Iwasawa decomposition, the work due to V. Balan and J.F. Dorfmeister in \cite{dorfmeister1} gives a construction of an Iwasawa decomposition for loop groups, which allows us factorize the harmonic maps in Lie groups and symmetric spaces \cite{dorfmeister2}, \cite{dorfmeister3}, \cite{dorfmeister4} and \cite{dorfmeister5}.

The purpose of this work is to factorize harmonic maps with values in a semisimple Lie groups with respect to its Iwasawa decomposition. Let us explain the idea of our work. Let $G$ be a connected, semisimple Lie groups with finite center and consider the Iwasawa decomposition gives by a compact Lie group $K$, abelian Lie group $A$ and nilpotent Lie group $N$ such that $G= K\cdot A \cdot N$. Let $M$ be a Riemannian manifold and  $F:M \rightarrow G$ a smooth map. We may write $F = F_k\cdot F_a \cdot F_n$, where $F_k$, $F_a$ and $F_n$ are projections of $F$ in $K$, $A$ and $N$, respectively.

The idea is to equip $G$ with a left invariant metric $<,>$ and to restrict it to left invariant metrics on $K, A$, $N$, because the projections $\pi_K: G \rightarrow K$, $\pi_A: G \rightarrow A$ and  $\pi_N: G \rightarrow N$ are Riemannian submersion. From this we prove that if $F$ is a harmonic map, then so are $F_k$, $F_a$ and $F_n$. On contrary, which is our main contribution, we use the structure of semisimple Lie group $G$ to show that if $F_k$, $F_a$ and $F_n$ are harmonic maps, then so is $F = F_k\cdot F_a \cdot F_n$( see Theorem \ref{maintheorem}).

We emphasize that the domain of harmonic maps are general Riemannian manifolds. %Furthermore, the structure of semisimple Lie group is the key to show Theorem \ref{maintheorem}.

The interest of this result is that it allow us to study the harmonic maps in structures more simple.  Here, as an application, we study the harmonic maps from $\mathbb{R}^n$ into $SL(2,\mathbb{R})$. In fact, we use the fundamental solution of Laplace's equation to give the domain and solution of harmonic maps in distinct cases $n=2$ and $n\geq 3$. In the case of $n=1$ we study the geodesics from a interval $I$ into $Sl(2, \mathbb{R})$.

The author wishes to express his tanks to Prof. Alexandre J. Santana and Prof. Josiney A. Souza for several helpful comments concerning to real, semisimple Lie algebra used in this work.

\section{A factorization of harmonic maps}

From now on we use freely the concepts and notations of  A. Knap \cite{knap}. For the convenience of the reader we repeat something about semisimples structure, thus making our exposition self-contained. Let $G$ be a connected, semisimple Lie group with finite center and $g$ its semisimple Lie algebra. From Iwasawa decomposition we know that there are a compact Lie group $K$, abelian Lie group $A$ and nilpotent Lie group $N$ such that
\[
  G = K \cdot A \cdot N.
\]
In correspondence there are Lie subalgebras $\mk$, $\ma$ and $\mn$ such that $\g = \mk \oplus \ma \oplus \mn$.

In the Lie algebra $\g$ there exists a Cartan involution $\theta: \g \rightarrow \g$ such that $\g = \mk \oplus \p$, where $\mk= \{ X \in \g : \theta(X) = X\}$ and $\p = \{ X \in \g : \theta(X) = -X\}$. Furthermore, the symmetric bilinear form
\[
  B_{\theta}(X,Y) = - B(X,\theta(Y))
\]
is positive defined.

Let $\ma$ be a maximal abelian subspace of $\p$. For $\lambda \in \ma^*$ we write
\[
  \g_{\lambda} = \{ X \in \g | \, \, ad(H)(X) = \lambda(H) X\,\, \textrm{for all}\,\, H \in \ma \}.
\]
If $\g_\lambda \neq 0$ and $\lambda \neq 0$, we call $\lambda$ a restrict root of $\g$. Denote by $\Sigma$ the set of restrict roots. Then
\[
  \g = \g_0 \oplus \bigoplus_{\lambda \in \Sigma}\g_{\lambda},
\]
where $\g_0 = \ma \oplus Z_{\mk}(\ma)$ and $\mathfrak{m} = Z_{\mk}(\ma) = \{ X \in \mk : [X, H] = 0 \textrm{ for all}\,\, H \in \ma \}$.

From now we assume that $G$ has a left invariant metric $<,>$. It is well know that the Levi-Civita connection associated to $<,>$ is given by
\[
  \alpha(X,Y) = \frac{1}{2}([X,Y] - ad^*(X)(Y) - ad^*(Y)(X))\ \ X,Y \, \in \, \g,
\]
where $ad^*$ is the adjoint map of $ad$ with respect to metric $<,>$. Let us denote by $\alpha_\s$ the symmetric part of $\alpha$, namely,
\[
  \alpha_\s(X,Y) = \frac{1}{2}(- ad^*(X)(Y) - ad^*(Y)(X))\ \ X,Y \, \in \, \g.
\]
Using the properties of semisimples Lie algebras we can rewrite the symmetric part of $\alpha$ in an appropriate way.
\begin{proposition}\label{levicivita}
  If $X,Y \in \g$, then
  \[
    \alpha_\s(X,Y) = [X^k,Y^p]+ [Y^k,X^p],
  \]
  where $X^{k}$ and $X^{p}$ are projections of $X$ in $\mk$ and $\p$, respectively.
\end{proposition}
\begin{proof}
We first observe that in a semisimple Lie algebra is true that $ad^*(X) = - ad(\theta(X))$ for $X \in \g$ (see Lemma 6.25 in \cite{knap}). Then for $X, Y \in \g$ we have
\begin{eqnarray*}
  \alpha_\s(X,Y)
  & = & \frac{1}{2}(-ad^*(X)(Y) -ad^*(Y)(X)) \\
  & = & \frac{1}{2}(ad(\theta(X))(Y)+ad(\theta(Y)(X))\\
  & = & \frac{1}{2}([\theta(X),Y] + [\theta(Y),X]).
\end{eqnarray*}
Writing $X = X^k + X^p$ and $Y = Y^k + Y^p$ where $X^k, Y^k \in \mk$ and $X^p, Y^p \in \p$ yields
\begin{eqnarray*}
  \alpha_\s(X,Y)
  & = & \frac{1}{2}([\theta(X^k + X^p),Y^k + Y^p] + [\theta(Y^k + Y^p),X^k + X^p])\\
  & = & \frac{1}{2}([\theta(X^k),Y^k]+[\theta(X^k),Y^p]+ [\theta(X^p),Y^k]+[\theta(X^p),Y^p]\\
  & + & [\theta(Y^k),X^k]+[\theta(Y^k),X^p] + [\theta(Y^p),X^k]+ [\theta(Y^p),X^p]).
\end{eqnarray*}
By definition of spaces $\mk$ and $\p$, it follows that
\begin{eqnarray*}
  \alpha_\s(X,Y)
  & = & \frac{1}{2}([X^k,Y^k]+[X^k,Y^p]+ [-X^p,Y^k]+[-X^p,Y^p]\\
  & + & [Y^k,X^k]+[Y^k,X^p] + [-Y^p,X^k]+ [-Y^p,X^p]).
\end{eqnarray*}
Canceling opposite terms we conclude that
\[
  \alpha_\s(X,Y) = [X^k,Y^p]+ [Y^k,X^p].
\]
\qed
\end{proof}

The proposition gain in interest if we realize that
\[
  \alpha_\s(Ad(g)X,Ad(g)Y) = Ad(g)\alpha_\s(X,Y)
\]
for $X,Y \in \g$ and $g \in G$.

In the sequel, we prove a technical lemma. The principal significance of the next lemma is that some cross terms in product $\mk \times \ma \times \mn$ vanishes with respect to symmetric part of $\alpha$.

\begin{lemma}\label{technical}
  For $X^k \in \mk$, $X^a \in \ma$ and $X^n \in \mn$ we have
  \[
    \alpha_\s(X^k,Ad(a) X^a)  = \alpha_\s(X^k,Ad(an)X^n) = \alpha_\s(X^a,Ad(n) X^n)  =  0,
  \]
  where $a \in A$ and $n \in N$.
\end{lemma}
\begin{proof}
  The basic idea is to use the root system associated to Lie algebra $\g$. In fact, let $\beta_{\ma} = \{H_1, \ldots, H_{m}\}$, $\beta_{\mathfrak{m}}=\{L_{1} ,\ldots, L_{n}\}$ and $\beta_{\lambda}= \{ \xi^{\lambda}_1, \ldots, \xi^{\lambda}_{n_{\lambda}}\}$ be basis of $\ma, \mathfrak{m}$ and $\g_{\lambda}$, respectively, such that $\beta = \beta_{\ma} \cup \beta_{\mathfrak{m}} \cup \bigcup_{\lambda \in \Sigma} \beta_{\lambda}$ is a orthogonal basis of $\g$ with respect to $B_\theta$. We begin showing that $\alpha_\s(X^k,Ad(a) X^a) = 0$ for $X^k \in \mk , X^a \in \ma$. Since $[X^k,Ad(a)X^a] \in \mathfrak{p} = \ma \oplus \mn$, it follows that
  \begin{eqnarray*}
    B_\theta([X^k,Ad(a)X^a], H_i)
    & = & - B(X^k,[X^a,\theta(H_i)]) = 0,
  \end{eqnarray*}
  where we used that $X^a,\theta(H_i) \in \ma$. Also, being $\mathfrak{n} = \sum_{\lambda>0} \g_\lambda$, for $\xi^{\lambda}_i \in \g_\lambda$ we see that
  \begin{eqnarray*}
    B_\theta([X^k,Ad(a)X^a], \xi^{\lambda}_i)
    & = & B_\theta(X^k,[X^a, \theta(\xi^{\lambda}_i)]) = 0,
  \end{eqnarray*}
  because $[X^a, \theta(\xi^{\lambda}_i)] \in \mathfrak{p}$. It follows that
  \begin{eqnarray*}
    \alpha_\s(X^k,Ad(a) X^a)
    & = & 0.
  \end{eqnarray*}
  Our next step is to show that $\alpha(X^k,Ad(an)X^n)=0$ for $ X^k \in \mk$, $ X^n \in \mn$.
  As $[X^k,Ad(an)X^n] \in \mathfrak{p} = \ma \oplus \mn$ we have
  \begin{eqnarray*}
    B_\theta([X^k,Ad(an)X^n], H_i)
    & = & - B_\theta(X^k,[Ad(an)X^n, \theta(H_i)])= 0,
  \end{eqnarray*}
  where we used that $[Ad(an)X^n, \theta(H_i)] \in \mathfrak{n}$. For $\xi^{\lambda}_i \in \g_\lambda$ we get
  \begin{eqnarray*}
    B_\theta([X^k,Ad(an)X^n], \xi^{\lambda}_i)
    & = & B_\theta(X^k,[Ad(an)X^n, \theta(\xi^{\lambda}_i)]) = 0
  \end{eqnarray*}
  because $[Ad(an)X^n, \theta(\xi^{\lambda}_i)] \in \mathfrak{p}$. In consequence,
  \begin{eqnarray*}
    \alpha_\s(X^k,Ad(an) X^n)
    & = & 0.
  \end{eqnarray*}
  Finally, Proposition \ref{levicivita} makes it is obvious that $\alpha_\s( X^a,Ad(n)X^n)= 0$ for $ X^a \in \ma$, $ X^n \in \mn$.
  \qed
\end{proof}

Our main theorem is now stated and proved.

\begin{theorem}\label{maintheorem}
  Let $G$ be a semisimple Lie group with a left invariant metric $<,>$. Assume that its Iwasawa decomposition is given by $G = K \cdot A \cdot N$. Furthermore, let $(M,g)$ be a Riemannian manifold and $F: M \rightarrow G$ a smooth map such that $F = F_k \cdot F_a \cdot F_n$, where $F_k:M \rightarrow K$,  $F_a: M \rightarrow A$ and $F_n: M \rightarrow N$ are projections of $F$ into $K$, $A$ and $N$, respectively. Then $F$ is a harmonic map if and only if and only if $F_k$, $F_a$ and $F_n$ are harmonic maps.
\end{theorem}
\begin{proof}
  Let $B_t$ a Brownian motion in $M$. Then $F(B_t)$ is a semimartingale in $G$ and $F(B_t)= F_k(B_t)\cdot F_a(B_t)\cdot F_n(B_t)$, where $F_k(B_t)$, $F_a(B_t)$ and $F_n(B_t)$ are semimartingales in $K, A$ and $N$, respectively. For simplicity, we denote $X_t = F(B_t)$, $X^k_t = F_k(B_t)$, $X^a_t = F_a(B_t)$ and $X^n_t = F_n(B_t)$. We also denote by $\nabla$ the Levi-Civita connection associated with $<,>$. Suppose now that $F$ is a harmonic map, then $X_t$ is a $\nabla$-martingale. Since $\pi_k$, $\pi_a$ and $\pi_n$ are Riemannian submersion, it follows that $X^k_t$, $X^a_t$ and $X^n_t$ are $\nabla$-martingales. It entails that $F_k$, $F_a$ and $F_n$ are harmonic maps.

  On the contrary, the idea is to study the decomposition of It\^o stochastic logarithm $\log{X_t}$ in terms of Iwasawa decomposition $\mk\oplus\ma \oplus \mn$. For this end, we use Proposition 5.1 in \cite{stelmastchuk} that says
  \[
    \log{X_t} = L(X_t) + \frac{1}{2}\int_0^t \alpha_\s(dL(X_s),dL(X_s)).
  \]
  Thus it is necessary to study the decomposition of stochastic logarithm $L(X_t)$. In fact, from Proposition 5 in \cite{hakim} we deduce that
  \begin{eqnarray*}
    L(X_t) & = & L(X^k_t \cdot X^a_t \cdot X^n_t) \\
           & = & \int_0^t Ad((X^a_s \cdot X^n_s)^{-1})dL(X^k_s) + L(X^a_t \cdot X^n_t)\\
           & = & \int_0^t Ad((X^a_s \cdot X^n_s)^{-1})dL(X^k_s) + \int_0^t Ad((X^n_s)^{-1})dL(X^a_s) + L(X^n_t).
  \end{eqnarray*}
  From this we compute $\alpha(dL(X_t),dL(X_t))$. In fact, using the $Ad$-invariance property with respect to $\alpha_\s$ given by Proposition \ref{levicivita} we obtain
  \begin{eqnarray*}
    \alpha_\s(dL(X_t),dL(X_t))
    & = & Ad((X^a_t \cdot X^n_t)^{-1})\alpha_\s(dL(X^k_t),dL(X^k_t))\\
    & + & Ad((X^n_t)^{-1})\alpha_\s(dL(X^a_t),dL(X^a_t))\\
    & + & \alpha_\s(dL(X^n_t),dL(X^n_t))\\
    & + & 2Ad((X^n_t)^{-1})Ad((X^a_t)^{-1})\alpha_\s(dL(X^k_t), Ad(X^a_t)dL(X^a_t))\\
    & + & 2Ad((X^n_t)^{-1})\alpha_\s(Ad(X^n_t)dL(X^n_t),dL(X^a_t))\\
    & + & 2Ad((X^n_s)^{-1})Ad((X^a_t)^{-1}\alpha_\s(dL(X^k_t),Ad((X^a_tX^n_t)dL(X^n_t)).
  \end{eqnarray*}
  From Lemma \ref{technical} it follows that $\log{X_t}$ is given by
  \begin{eqnarray*}
    \log{X_t}
    & = &  \int_0^t Ad((X^a_s \cdot X^n_s)^{-1})dL(X^k_s) + \int_0^t Ad((X^n_s)^{-1})dL(X^a_s) + L(X^n_t) \\
    & + & \frac{1}{2} \int_0^t Ad((X^a_s \cdot X^n_s)^{-1})\alpha_\s(dL(X^k_s),dL(X^k_s))\\
    & + & \frac{1}{2} \int_0^t Ad((X^n_s)^{-1})\alpha_\s(dL(X^a_s),dL(X^a_s))\\
    & + & \frac{1}{2} \int_0^t \alpha_\s(dL(X^n_s),dL(X^n_s))\\
  \end{eqnarray*}
  From Proposition 5.1 in \cite{stelmastchuk} we conclude that
  \begin{eqnarray*}
    \log{X_t}  = \int_0^t Ad((X^a_s \cdot X^n_s)^{-1})d\log{(X^k_s)} + \int_0^t Ad((X^n_s)^{-1})d\log{(X^a_s)} + \log{(X^n_t)}.
  \end{eqnarray*}
  Assume now that $F_k$, $F_a$ and $F_n$ are harmonic maps. Consequently, $X^k_t$, $X^a_t$ and $X^n_t$ are $\nabla$-martingales. Hence $\log{X_t}$ is a local martingale in $\g$. Now, Corollary 3.8 in \cite{stelmastchuk} shows that $X_t$ is a $\nabla$-martingale in $G$. This gives that $F$ is a harmonic map. \qed
\end{proof}

As a particular case, since geodesics are harmonic maps, we have the following Corollary.

\begin{corollary}\label{geodesics}
  Under assumptions of Theorem above, if moreover $I$ is an interval of $\mathbb{R}$ and if $\gamma: I \rightarrow G$ is a smooth curve such that $\gamma = \gamma_k \cdot \gamma_a \cdot \gamma_n$, then $\gamma$ is a geodesic if and only if $\gamma_k$, $\gamma_a$ and $\gamma_n$ are geodesics.
\end{corollary}

\subsection{Harmonic maps  with values in $SL(2,\mathbb{R})$}

In this section, we work with harmonic maps with values in $SL(2,\mathbb{R})$. We begin by introducing a condition that characterize harmonic maps with a values in a Lie Group. Let $M$ be a Riemannian manifold and $G$ a Lie group with a left invariant metric $<,>$. From Example 5.1 in \cite{stelmastchuk} we see that a smooth map $F:M \rightarrow G$ is harmonic if and only if
\begin{equation}\label{harmonicliegroup}
  d^{*}\omega_{F}  - \sum_{i=1}^{n}ad(\omega_{F}(e_{i}))^*(\omega_{F}(e_{i})) = 0,
\end{equation}
where $d^*$ is the codifferential of the Riemannian manifold $M$ and $\omega$ is the Maurer-Cartan form on $G$.

We now introduce the Iwasawa decomposition of $Sl(2,{\Bbb R})$. In fact, $Sl(2,{\Bbb R})$ can be written as the product $SO(2, {\Bbb R}) \cdot A \cdot N$ where
\[
  SO(2, {\Bbb R})
  =
  \left\{
  \left(
  \begin{array}{cc}
    \cos\theta & -\sin\theta\\
    \sin\theta & \cos\theta
  \end{array}
  \right), \theta \in [-\pi, \pi)
  \right\}
\]
\[
  A=
  \left\{
  \left(
  \begin{array}{cc}
    r & 0\\
    0& 1/r
  \end{array}
  \right), r>0
  \right\}, \ \
  \textrm{and}\ \
  N=
  \left\{
  \left(
  \begin{array}{cc}
    1 & x\\
    0& 1
  \end{array}
  \right), x \in {\Bbb R}
  \right\}.
\]
Thus for each $g \in SL(2, {\Bbb R})$ there exist $\theta \in [-\pi, \pi)$, $r>0$ and $x \in {\Bbb R}$ such that
\[
  g=\left(
  \begin{array}{cc}
    \cos\theta & -\sin\theta\\
    \sin\theta & \cos\theta
  \end{array}
  \right)
  \left(
  \begin{array}{cc}
    r & 0\\
    0 & 1/r
  \end{array}
  \right)
  \left(
  \begin{array}{cc}
    1 & x\\
    0 & 1
  \end{array}
  \right)
  =
  \left(
  \begin{array}{cc}
    r\cos\theta& xr\cos\theta - (1/r)\sin\theta\\
    r\sin\theta& xr\sin\theta + (1/r)\cos\theta
  \end{array}
  \right).
\]

Assume that $SL(2,\mathbb{R})$ is equipped with a left invariant metric $<,>$. Restricting the metric $<,>$ to $so(n, \mathbb{R})$, $\ma$ and $\mn$ we have left invariant metrics on $SO(n, \mathbb{R})$, $A$ and $N$. Since that $\mathfrak{so}(2,\mathbb{R})$, $\mathfrak{a}$ and $\mathfrak{n}^+$ are $1$-dimensional, we have that $ad^*_k = ad^*_a = ad^*_n = 0$, where $ad^*_k$, $ad^*_a$ and $ad^*_n$ are the adjoint operators of $ad_k$, $ad_a$ and $ad_n$ on $\mk, \ma$ and $\mn$, respectively.

Let $F: \mathbb{R}^n \rightarrow SL(2,\mathbb{R})$ be a smooth map. It is immediate that $F = F_1 \cdot F_2 \cdot F_3$, where $F_1: \mathbb{R}^n \rightarrow SO(2,\mathbb{R})$, $F_2: \mathbb{R}^n \rightarrow A$ and $F_3: \mathbb{R}^n \rightarrow N$. Thus there exist functions $f: \mathbb{R}^n \rightarrow [-\pi, \pi)$, $g:\mathbb{R}^n \rightarrow (0, \infty)$ and $h:\mathbb{R}^n \rightarrow \mathbb{R}$ such that
\[
  F_{1}(x)
  = \left(
  \begin{array}{cc}
    F_1^1(x) & -F_1^2(x)\\
    F_1^2(x) & F_1^1(x)
  \end{array}
  \right), \ \
  F_{2}(x) =
  \left(
  \begin{array}{cc}
    g(x) & 0\\
    0& \frac{1}{g(x)}
  \end{array}
  \right),
  \ \
  F_3(x) =
  \left(
  \begin{array}{cc}
    1 & h(x)\\
    0& 1
  \end{array}
  \right),
\]
where $F^1_1(x)=\rm cos(f(x))$, $F^1_2(x) =\rm sin(f(x))$ and $(F_1^1(x))^2+ (F_1^2(x))^2 = 1$.

From Theorem \ref{maintheorem} it follows that $F$ is a harmonic map with respect to metric $<,>$ if and only if so are $F_1$, $F_2$ and $F_3$. From equation (\ref{harmonicliegroup}) we see that $F_1$, $F_2$ and $F_3$ are harmonic maps if and only if
\begin{equation*}
  d^*\omega_{F_1} = 0 , \ \ d^*\omega_{F_2} = 0, \ \ \textrm{and} \ \ d^*\omega_{F_3} = 0,
\end{equation*}
respectively. Using the fact that $\omega(g) = g^{-1} dg$ we can written, respectively, these equations as
\[
  F_1(x)\Delta F_2(x) - F_2(x)\Delta F_1(x) = 0, \ \ \Delta (g(x)g(x)) = 0,\ \ \textrm{and} \ \Delta( h(x)h(x)) = 0.
\]

From the first equation we have $\Delta f(x)= 0$. Hence from the fundamental solution of Laplace's equation( see page 22 in \cite{evans}) we obtain
\begin{equation}\label{solution}
  f(x) = g^2(x) = h^2(x)
  \left\{
  \begin{array}{cc}
    -\frac{1}{2\pi} log(|x|)&, n=2\\
    \frac{1}{(n(n-2))\alpha(n)} \frac{1}{|x|^{n-2}}&, n \geq 3
  \end{array}
  \right.,
\end{equation}
where $\alpha(n)$ is the volume of the unit ball in $\mathbb{R}^n$ for $n \geq 3$.

Our next step is to study the domain of $F$. Firstly, suppose that $n=2$. Since equation (\ref{solution}) blows up at 0, we need to isolate the singularity inside a small ball. So fix $\epsilon>0$. Observing the image of functions $f, g$ and $h$ we deduce that the domain of $F$, denoted by $D_F$, is $D_F = B_{\mathbb{R}^2}(0,1) - B_{\mathbb{R}^2}(0,\epsilon)$. In consequence, a smooth map $F: D_F\subset \mathbb{R}^2 \rightarrow SL(2,\mathbb{R})$ is harmonic if it is written by
\[
  F(x)=
  \left(
  \begin{array}{cc}
    \sqrt{t(x)}\cos(t(x)) & t(x)\sqrt{t(x)}\cos(t(x))-\frac{\sin(t(x))}{\sqrt{t(x)}}\\
    \sqrt{t(x)}\sin(t(x)) & t(x)\sqrt{t(x)}\sin(t(x))+\frac{\cos(t(x))}{\sqrt{t(x)}}
  \end{array}
  \right),
\]
where $t:D_F \rightarrow \mathbb{R}$ is a function given by $t(x) = -\frac{1}{2\pi} log(|x|)$.

In the case of $n \geq 3$, we take the domain $D_F$ as $\mathbb{R}^n - \overline{B_{\mathbb{R}^n}(0,1/(\pi n(n-2)\alpha(n)))} $. Hence a smooth map $F: D_F \rightarrow SL(2, \mathbb{R})$ is harmonic if $F$ is written as
\[
  F(x)=
  \left(
  \begin{array}{cc}
    \sqrt{s(x)}\cos(s(x)) & s(x)\sqrt{s(x)}\cos(s(x))-\frac{\sin(s(x))}{\sqrt{s(x)}}\\
    \sqrt{s(x)}\sin(s(x)) & s(x)\sqrt{s(x)}\sin(s(x))+\frac{\cos(s(x))}{\sqrt{s(x)}}
  \end{array}
  \right),
\]
where $s:D_F \rightarrow \mathbb{R}$ is a function given by $s(x) = \frac{1}{(n(n-2))\alpha(n)} \frac{1}{|x|^{n-2}}$.

In the case of $n=1$, geodesics are treat instead of harmonic maps. Let $I$ be an interval of $\mathbb{R}$ and $\gamma: I \rightarrow SL(2,\mathbb{R})$ a smooth curve such that $\gamma = \gamma_1 \cdot \gamma_2 \cdot \gamma_3$, where $\gamma_1: I \rightarrow SO(2,\mathbb{R})$, $\gamma_2: I \rightarrow A$ and $\gamma_3: I \rightarrow N$ are smooth curves. From Corollary \ref{geodesics} it follows that $\gamma$ is geodesic if and only if are so $\gamma_1$, $\gamma_2$ and $\gamma_3$.

Our next step is to found a solution for this problem. Let $\gamma:I \rightarrow SL(2,\mathbb{R})$ be a smooth curve with $\gamma(0) = Id$ and $I$ an interval. We beginning by writing
\[
  \gamma_{1}(t)
  = \left(
  \begin{array}{cc}
    x(t) & -y(t)\\
    y(t) & x(t)
  \end{array}
  \right), \ \
  \gamma_{2}(t) =
  \left(
  \begin{array}{cc}
    a(t) & 0\\
    0& \frac{1}{a(t)}
  \end{array}
  \right),
  \ \
  \gamma_3(t) =
  \left(
  \begin{array}{cc}
    1 & n(t)\\
    0& 1
  \end{array}
  \right),
\]
where $x(t)=\rm cos(\eta(t))$, $y(x) = \rm sin(\eta(t))$ and $(\gamma_1^1(t))^2+ (\gamma_1^2(t))^2 = 1$. From equation (\ref{harmonicliegroup}) we see that $\gamma_1$, $\gamma_2$ and $\gamma_3$ are harmonic maps if and only if
\begin{equation*}
  \frac{d}{dt}(\gamma_1^{-1}\dot{\gamma}_1) = 0, \ \ \frac{d}{dt}(\gamma_2^{-1}\dot{\gamma}_2) = 0, \ \  \frac{d}{dt}(\gamma_3^{-1}\dot{\gamma}_3) = 0,
\end{equation*}
respectively. Using the fact that $\omega(g) = g^{-1} dg$ we can written, the first equation gives the following differential equation system
\begin{eqnarray*}
  \dot{x}(t)\dot{x}(t)+ x(t)\ddot{x}(t) + \dot{y}(t)\dot{y}(t)+ \ddot{y}(t) & = & 0\\
  \dot{x}(t)\ddot{y}(t)-\dot{y}(t)\ddot{x}(t) & = & 0,
\end{eqnarray*}
which as solution $\ddot{\eta}(t) = 0$. Thus under initial conditions $\gamma_1(0) =\gamma_2(0) = \gamma_3(0)= Id$ we obtain as solutions
\[
  \gamma_{1}(t)
  = \left(
  \begin{array}{cc}
    \cos(at) & -\sin(at)\\
    \sin(at) & \cos(at)
  \end{array}
  \right), \ \
  \gamma_{2}(t) =
  \left(
  \begin{array}{cc}
    e^{ct} & 0\\
    0& \frac{1}{e^{ct}}
  \end{array}
  \right), \ \
  \gamma_3(t) =
  \left(
  \begin{array}{cc}
    1 & kt\\
    0& 1
  \end{array}
  \right),
\]
which give
\[
  \gamma(t)=
  \left(
  \begin{array}{cc}
    e^{ct}\cos(at) & kte^{ct}\cos(at)-\frac{\sin(at)}{e^{ct}}\\
    e^{ct}\sin(at) & kte^{ct}\sin(at)+\frac{\cos(at)}{e^{ct}}
  \end{array}
  \right),
\]
where  $a,c,e \in \mathbb{R}$ are constants given by initial condition $\dot{\gamma}(0)$ and $t \in \mathbb{R}$.


\begin{thebibliography}{20}

\bibitem{dorfmeister1} Balan, V., Dorfmeister, J. F. {\it Birkhoff decompositions and Iwasawa decompositions for loop groups}. Tohoku Math. J. (2) 53 (2001), no. 4, 593–615.

\bibitem{dorfmeister2} Balan, V., Dorfmeister, J. F., {\it Harmonic maps into general symmetric spaces via loop groups}. (English summary) Recent advances in geometry and topology, 49–64, Cluj Univ. Press, Cluj-Napoca, 2004.

\bibitem{dorfmeister3} Balan, V., Dorfmeister, J.F., {\it A Weierstrass-type representation for harmonic maps from Riemann surfaces to general Lie groups}. (English summary) Dedicated to Professor Constantin Udriste. Balkan J. Geom. Appl. 5 (2000), no. 1, 7–37.


\bibitem{burstall} Burstall,F.E., Guest, M.A.,{\it Harmonic two-spheres in compact symmetric spaces, revisited}. Math. Ann. 309 (1997) 541-572.

\bibitem{dorfmeister4} Dorfmeister, J. F., Pedit, F.; Wu, H., {\it Weierstrass type representation of harmonic maps into symmetric spaces}. Comm. Anal. Geom. 6 (1998), no. 4, 633–668.

\bibitem{dorfmeister5} Dorfmeister J.F.,  Inoguchi J-I, Kobayashi S., {\it A loop group method for affine harmonic maps into Lie groups}. http://arxiv.org/abs/1405.0333

\bibitem{evans} Evans, L.C., {\it Partial differential equations.} Second edition. Graduate Studies in Mathematics, 19. American Mathematical Society, Providence, RI, 2010. xxii+749 pp.

\bibitem{hakim} Hakim-Dowek, M., and L\'epingle, D., {\it L'exponentielle Stochastique de Groupes de Lie}. Lectures Notes in Mathematics, 1204, 1986, p. 352-374.

\bibitem{knap} Knap, Anthony W. {\it Lie Groups beyond an introduction}. Second edition, Progress in Mathematics, Volume 140, Birkh\"{a}user, 2002.

\bibitem{uhlenbeck}  Uhlenbeck,K.K., {\it Harmonic maps into Lie groups (Classical solutions of the chiral model)}. J. Diff. Geom. 30 (1989), 1–50

\bibitem{stelmastchuk} Stelmastchuk, S. N.,{\it The Ito exponential on Lie groups}. Int. J. Contemp. Math. Sci., 8 (2013), no. 5-8,
307–326.

\end{thebibliography}
\end{document}